\documentclass[12pt]{amsart}
\usepackage[pdftex]{hyperref}
\usepackage{graphicx}
\usepackage{tikz}
\usepackage{amssymb}
\usepackage{multirow}
\allowdisplaybreaks
\def\picwidth{4.8in}

\def\GL{\mathrm{GL}}
\def\Map{\mathrm{Map}}
\def\Out{\mathrm{Out}}
\def\Homeo{\mathrm{Homeo}}
\def\defn#1{\emph{#1}}

\def\Ftilde{\tilde{F}}
\def\Ghat{\hat{G}}
\def\Gtilde{\tilde{G}}

\def\atilde{\tilde{a}}
\def\btilde{\tilde{b}}
\def\ctilde{\tilde{c}}
\def\dtilde{\tilde{d}}
\def\ahat{\hat{a}}
\def\bhat{\hat{b}}
\def\chat{\hat{c}}
\def\dhat{\hat{d}}
\def\ghat{\hat{g}}
\def\xtilde{\tilde{x}}
\def\ytilde{\tilde{y}}
\def\Z{\mathbb{Z}}

\def\R{\mathbb{R}}

\def\Ker{\mathop{\mathrm{Ker}}\nolimits}

\let\iso\cong
\def\cong{\equiv}

\def\Aut{\mathop{\mathrm{Aut}}\nolimits}

\def\gend#1{\langle #1\rangle}
\def\sset{\subseteq}

\newtheorem{theorem}{Theorem}
\newtheorem{lemma}[theorem]{Lemma}
\newtheorem{corollary}[theorem]{Corollary}

\theoremstyle{remark}
\newtheorem{example}[theorem]{Example}

\numberwithin{theorem}{section}
\numberwithin{equation}{section}

\begin{document}

\title{Most big mapping class groups fail the Tits alternative}
\author{Daniel Allcock}
\thanks{Supported by Simons Foundation 
Collaboration Grant 429818}
\address{Department of Mathematics\\University 
of Texas, Austin}
\email{allcock@math.utexas.edu}
\urladdr{http://www.math.utexas.edu/\textasciitilde allcock}
%
\date{April 30, 2020}

\begin{abstract}
    Let $X$ be a surface, possibly with boundary. 
    Suppose it has infinite genus or infinitely many punctures,
    or a closed subset which is a disk with a Cantor set
    removed from its interior.
    For example, $X$ could be any surface of infinite type with
    only  finitely many 
    boundary components.
    We prove that the
    mapping class group of~$X$ 
    does not satisfy the Tits Alternative.  That is,
    $\Map(X)$ contains a finitely generated subgroup that 
    is not virtually solvable and contains no nonabelian free group.
\end{abstract}

\maketitle

\section{Introduction}
\label{SecIntro}

\noindent
Lanier and Loving gave examples of 
big mapping class groups that do not satisfy
the Tits Alternative, and asked whether 
the same holds for every
big mapping class group
\cite[Question 6]{LL}.  We show that few if any big mapping
class groups satisfy it:

\begin{theorem}
    \label{ThmMain}
    Suppose $X$ is a surface, possibly nonorientable and
    possibly with boundary. Also suppose that it
     satisfies one of the following:
    \begin{enumerate}
    \item
        \label{CaseGenus}
        $X$ has infinite genus;
    \item
        \label{CasePunctures}
        $X$ has infinitely many punctures;
    \item
        \label{CaseCantor}
        $X$ contains a closed subset homeomorphic to $D^2-C$, where $C$ is
            a  Cantor set in the interior of the $2$-disk.
    \end{enumerate}
    Then its mapping class group $\Map(X)$ does not satisfy the
    Tits Alternative.  That is, 
    $\Map(X)$ 
    has a finitely generated
    subgroup~$\Gtilde$, which 
    contains no nonabelian
    free group and no finite-index solvable
    subgroup.
\end{theorem}

\begin{corollary}
    \label{CorMain}
    Suppose $X$ is a surface of infinite type, with
    only finitely many boundary components.  Then  $\Map(X)$ 
    does not satisfy the Tits Alternative.
\end{corollary}

For us, every
surface $X$ is connected and second countable, and may have boundary.
When $X$ is non-orientable, its \emph{mapping class group} $\Map(X)$
is defined as the group of self-homeomorphisms of~$X$
that fix the boundary pointwise, modulo isotopies that
fix the boundary pointwise.  When $X$ is orientable we
use the same definition except that the self-homeomorphisms
are required to preserve orientation.
A surface~$X$ has \emph{finite 
type} or \emph{infinite type} according to whether its fundamental
group is finitely generated or not.  
A \emph{small} resp.\ \emph{big} mapping class group means the
mapping class group of a finite-type resp.\ infinite-type surface.
See \cite{FarbMargalit} 
for general background on mapping class groups,
and \cite{AV} for a survey of recent work on big mapping class groups.
Although $\Map(X)$ is a topological group, its topology
will play no role in this paper.

The \emph{Tits Alternative} is a famous property of $\GL_n(k)$
over any field~$k$,
discovered
by Tits \cite{Tits}.  
Namely, every finitely generated
subgroup  is either ``big'' 
(contains
a nonabelian free group) 
or ``small'' (contains  a finite-index solvable
subgroup).  
The point is that there is no ``medium''.
One expresses this by
saying that $\GL_n(k)$ satisfies the Tits Alternative.
The property makes sense with any 
group in place of $\GL_n(k)$, and if it holds then
we say that group satisfies the Tits Alternative.

Determining whether various groups satisfy the Tits Alternative
 is a major
thread in geometric group theory.
For example, it holds 
for the outer automorphism group of a finite-rank free group
\cite{BestvinaFeighnHandel1}\cite{BestvinaFeighnHandel2}, 
and for many Artin groups \cite{MartinPrzytycki}.
It also holds 
for small mapping class groups,
in fact Ivanov \cite{Ivanov} and McCarthy \cite{McCarthy} independently 
proved that 
these groups satisfy a ``strong Tits alternative''.
This is
defined the same way except with solvable  
replaced by abelian.  

Lanier
and Loving \cite{LL} showed that the mapping class group
of an infinite-type open surface cannot have this stronger
property.  They also gave examples of big mapping
class groups that do not satisfy even the classical Tits alternative.
The first such example, not noted at the time,
is probably due to Funar and
Kapoudjian.  They exhibited a big mapping class
group that contains Thompson's
group \cite[Prop.~2.4]{FK}\cite[Prop.~5.30]{AV}.
Further examples come from the work of
Aougab, Patel and Vlamis \cite{APV}, who constructed
surfaces~$X$ such that  every countable
group embeds in~$\Map(X)$.

The key idea in the proof of  Theorem~\ref{ThmMain} is
to construct $\Gtilde$
together with a surjection $\Gtilde\to G$, where $G$ is
Grigorchuck's group.
We do this in such a way that the kernel is abelian, which 
allows us
to transfer to~$\Gtilde$ the fact
that $G$ does not satisfy the Tits
Alternative.  
We review the essential properties of~$G$ in Section~\ref{SecGrigorchuk}.
Section~\ref{SecCantorSet} gives our
construction in case \eqref{CaseCantor} of the theorem, and Section~\ref{SecPuncturesOrHandles}
gives it for the other two cases.  Also see Section~\ref{SecPuncturesOrHandles} for
the definition of a puncture.  Section~\ref{SecMainTheorem}
shows that these three cases
are enough for Corollary~\ref{CorMain}, and makes a few remarks
about the cases left open.

We are grateful to Lanier and Loving for helpful correspondence
and discussions, for many references,
and for asking the question that inspired this work.

\section{Grigorchuk's group}
\label{SecGrigorchuk}

\noindent
In this section we recall Grigorchuk's famous group
and some of its many remarkable properties;
see \cite{Grigorchuk} for background. 
The properties we need are that
it is a group of automorphisms of a rooted
binary tree, and is 
finitely generated, infinite and torsion.  
Sadly,
we do not use its most fascinating property, that it
has intermediate growth.

We define $T$ as the tree whose vertices are the finite 
binary
words (sequences in $\{0,1\}$). 
We indicate the empty
word by $\emptyset$. 
The edges of $T$ are the following:
each word $w$ is joined to $w0$ and $w1$, which are called
the left and right \defn{children} of~$w$ respectively.  
The $n$th 
\defn{bit} of~$w$ means the $n$th term of the sequence~$w$.
This only makes sense when $w$ has at least~$n$ terms.

Grigorchuk's group $G$ is defined as the subgroup of $\Aut(T)$
generated by four specific transformations $a,b,c,d$.  Each of these
fixes the root $\emptyset$ of~$T$.  At other vertices the definitions
are recursive, expressed in terms of an arbitrary word $w$ 
(possibly empty):
\begin{center}
    \setlength{\tabcolsep}{0pt}
    \begin{tabular}{rlrlrlrl}
        $a(0w)$&${}=1w$
        &
        \qquad
        $b(0w)$&${}=0a(w)$
        &
        \qquad
        $c(0w)$&${}=0a(w)$
        &
        \qquad
        $d(0w)$&${}=0w$
        \\
        $a(1w)$&${}=0w$
        &
        $b(1w)$&${}=1c(w)$
        &
        $c(1w)$&${}=1d(w)$
        &
        $d(1w)$&${}=0b(w)$
    \end{tabular}
\end{center}

Flipping a bit in a word means changing that bit from $0$ to~$1$
or vice-versa.  So  $a$ acts by flipping
the first bit of  each nonempty word.
We can also describe $b,c,d$ in terms of bit-flipping.  Namely, $b$ acts on a word~$w$ by 
flipping the bit after the first $0$ in~$w$, just if the number of
initial $1$'s in~$w$ is $0$ or~$1$ mod~$3$.  (If $w$ has no
$0$'s, or no bits after its first~$0$, then $b$ fixes~$w$.)
For $c$ we use the same definition but with ``$0$ or~$1$'' replaced by
``$0$ or~$2$'', and similarly for $d$ using
``$1$ or~$2$''.
An induction justifies these descriptions,
pictured  in Figure~\ref{FigGrigorchuk}.

\def\triangleheight{2}
\def\trianglehalfwidth{.22}
\definecolor{treegray}{gray}{.8}%
\def\shadedtriangle#1#2{%
    \fill[treegray](#1,#2)
    --(-\trianglehalfwidth+#1,-\triangleheight+#2)
    --(\trianglehalfwidth+#1,-\triangleheight+#2)
    --cycle
    ;}
\def\noderadius{1.5pt}%
\def\leftsubtreefixed#1#2{%
    \fill(#1,#2)circle(\noderadius);       
    \shadedtriangle{-1+#1}{-1+#2}
    \draw[thick](#1,#2)--(-1+#1,-1+#2)
    --(-1-\trianglehalfwidth+#1,-1-\triangleheight+#2)
    --(-1+\trianglehalfwidth+#1,-1-\triangleheight+#2)--(-1+#1,-1+#2)
    ;}
\def\swapheight{2}
\def\swaphalfwidth{.28}
\def\arcradius{\swapheight}
\def\archalfangle{6}
\def\swappedtriangles#1#2{%
    \shadedtriangle{-\swaphalfwidth+#1}{-\swapheight+#2}
    \shadedtriangle{\swaphalfwidth+#1}{-\swapheight+#2}
    \draw[->](#1,-\arcradius+#2) arc (-90:-90+\archalfangle:\arcradius);
    \draw[->](#1,-\arcradius+#2) arc (-90:-90-\archalfangle:\arcradius);
    \draw[thick](-\swaphalfwidth+#1,-\swapheight+#2)
    --(-\swaphalfwidth+\trianglehalfwidth+#1,-\swapheight-\triangleheight+#2)
    --(-\swaphalfwidth-\trianglehalfwidth+#1,-\swapheight-\triangleheight+#2)
    --(-\swaphalfwidth+#1,-\swapheight+#2)
    --(#1,#2)
    --(\swaphalfwidth+#1,-\swapheight+#2)
    --(\swaphalfwidth+\trianglehalfwidth+#1,-\swapheight-\triangleheight+#2)
    --(\swaphalfwidth-\trianglehalfwidth+#1,-\swapheight-\triangleheight+#2)
    --(\swaphalfwidth+#1,-\swapheight+#2)
    ;}
\def\leftsubtreeswaps#1#2{%
    \fill(#1,#2)circle(\noderadius);
    \fill(-1+#1,-1+#2)circle(\noderadius);
    \draw[thick](0+#1,0+#2)--(-1+#1,-1+#2);
    \swappedtriangles{-1+#1}{-1+#2}%
    }
\def\rightedge{%
    \draw[thick](0,0)--(6.2,-6.2);
    \draw[thick,dashed](6.1,-6.1)--(7.2,-7.2);
    }
\def\xGrigorchuk{.44cm}
\def\yGrigorchuk{.22cm}
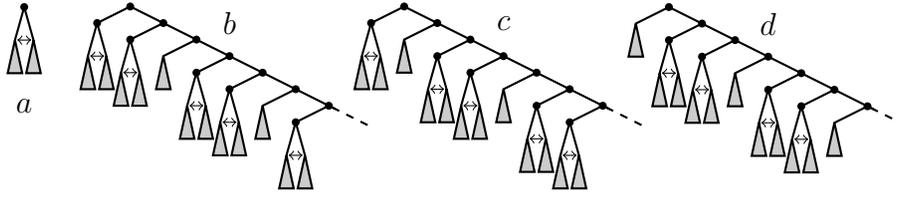
\begin{figure}%
    \begin{tikzpicture}[x=\xGrigorchuk,y=\yGrigorchuk]
        \draw[white](0,0)--(0,-7-\swapheight-\triangleheight);
        \fill(0,0)circle(\noderadius);       
        \swappedtriangles{0}{0}
        \node at (0,-6) {$a$};
    \end{tikzpicture}
    \kern10pt
    \begin{tikzpicture}[x=\xGrigorchuk,y=\yGrigorchuk]
        \draw[white](2,0)--(2,-7-\swapheight-\triangleheight);
        \rightedge
        \leftsubtreeswaps{0}{0}%
        \leftsubtreeswaps{1}{-1}%
        \leftsubtreefixed{2}{-2}%
        \leftsubtreeswaps{3}{-3}%
        \leftsubtreeswaps{4}{-4}%
        \leftsubtreefixed{5}{-5}%
        \leftsubtreeswaps{6}{-6}%
        \node at (3,-1) {$b$};
    \end{tikzpicture}
    \kern-10pt
    \begin{tikzpicture}[x=\xGrigorchuk,y=\yGrigorchuk]
        \draw[white](2,0)--(2,-7-\swapheight-\triangleheight);
        \rightedge
        \leftsubtreeswaps{0}{0}%
        \leftsubtreefixed{1}{-1}%
        \leftsubtreeswaps{2}{-2}%
        \leftsubtreeswaps{3}{-3}%
        \leftsubtreefixed{4}{-4}%
        \leftsubtreeswaps{5}{-5}%
        \leftsubtreeswaps{6}{-6}%
        \node at (3,-1) {$c$};
    \end{tikzpicture}
    \kern-10pt
    \begin{tikzpicture}[x=\xGrigorchuk,y=\yGrigorchuk]
        \draw[white](2,0)--(2,-7-\swapheight-\triangleheight);
        \rightedge
        \leftsubtreefixed{0}{0}%
        \leftsubtreeswaps{1}{-1}%
        \leftsubtreeswaps{2}{-2}%
        \leftsubtreefixed{3}{-3}%
        \leftsubtreeswaps{4}{-4}%
        \leftsubtreeswaps{5}{-5}%
        \leftsubtreefixed{6}{-6}%
        \node at (3,-1) {$d$};
    \end{tikzpicture}
    \caption{Generators for Grigorchuk's group~$G$.  Each $2$-headed 
    arrow represents the
    bodily exchange of the two indicated subtrees, drawn as shaded triangles.
    The other shaded triangles indicate pointwise-fixed subtrees.}
    \label{FigGrigorchuk}
\end{figure}

If $g\in\{a,b,c,d\}$ fixes a vertex~$v$ but exchanges the children
of~$v$, then we call $v$ a \defn{swap vertex} of~$g$.
One can also define the swap vertices
as the boundary points
of the fixed-point set
in~$T$.
The only swap vertex of~$a$ is $\emptyset$.  The swap vertices
of~$b$ are the words $1\cdots10$, where the number of $1$'s is $0$ or~$1$ mod~$3$.  
And similarly
for $c$ resp.\ $d$, with ``$0$ or~$1$'' replaced by ``$0$ or~$2$'' resp.\ ``$1$ or $2$''.

\begin{lemma}
    \label{LemGrigorchuk}
$G$ is not virtually solvable, and contains no nonabelian free group.
\end{lemma}

\begin{proof}
    The second claim follows from the fact that $G$ is a torsion group.
    For the first, suppose $G$ contains a solvable subgroup~$F$ of finite
    index.  Having finite index in the finitely generated group~$G$,
    $F$ is also finitely generated.  Every finitely generated solvable torsion
    group is finite.  (The abelianization
    is finite, so
    the derived subgroup is finitely generated, hence finite
    by induction on solvable length.)  So $F$ is finite, which forces~$G$
    to be finite, which it is not.
\end{proof}

\section{If \texorpdfstring{$X$}{X} contains \texorpdfstring{$D^2-\hbox{(Cantor set)}$}{a disk minus a Cantor set}}
\label{SecCantorSet}

\noindent
In this section we suppose that $X$ has a closed subset
homeomorphic to $D^2-C$, 
where $D^2$ is the $2$-disk and $C$ is a Cantor
set in the interior of $D^2$.  
%
%
We will 
construct a subgroup $\Gtilde$ of $\Map(X)$, supported on~$D^2-C$,
which is not virtually solvable and contains no nonabelian free
group.
The idea is to interpret
the generators 
of Grigorchuk's group~$G$ as
automorphisms 
of~$X$, 
and define $\Gtilde$ as the subgroup of $\Map(X)$ generated by
the mapping classes represented by these automorphisms.

We will work with the following  description of $D^2-C$ as an identification
space, obtained by gluing pairs of pants together in the pattern of the
tree~$T$.
Fix
a pair of pants~$P$, meaning a sphere
minus the interiors of three pairwise disjoint closed disks.
We suppose that the components of $\partial P$ 
are called (in some order)
the waist and the left and right cuffs of~$P$. 
We fix a self-homeomorphism $\sigma$ (for \emph{s}wap)
of $P$, which fixes the waist pointwise and exchanges the cuffs, 
such that $\sigma^2$ is the identity on each cuff.
We choose a homeomorphism from the waist to~$S^1$.  We choose
homeomorphisms from the cuffs to~$S^1$
which are compatible with each other,
in the sense that 
they are exchanged by~$\sigma$.

We define $W$ (for binary \emph{W}\kern-.15em ords) as the set of vertices of~$T$,
equipped with the discrete topology.  We equip $P\times W$ with the product topology.
For each $w\in W$ we abbreviate $P\times\{w\}$ as~$P_w$.
We transfer  our labeling (as waist or left/right cuff)
of the components of~$\partial P$
to those of~$\partial P_w$.  

We define an equivalence relation~$\sim$ on $P\times W$ by 
gluing the left resp.\ right cuff of each $P_w$ to the waist
of~$P_{w0}$ resp.\ $P_{w1}$.  Formally, we 
declare that
for each $w\in W$ we have $(l,w)\sim(l',w0)$ and $(r,w)\sim(r',w1)$,
where $l$ resp.\ $r$ lies in the left resp.\ right cuff of~$P$,
and $l'$ resp.\ $r'$ is the corresponding point on the waist of $P$.
(We fixed homeomorphisms from the waist and  cuffs to~$S^1$.  
Combining them
gives
homeomorphisms from the
cuffs to the waist, which we take as the definitions of $l\mapsto l'$
and $r\mapsto r'$.)
We write $X_\emptyset$ for $(P\times W)/{\sim}$, equipped with the
quotient space topology.  It is standard that $X_\emptyset\iso D^2-C$, so
we regard $X_\emptyset$ as a subsurface of~$X$.  

The curious notation $X_\emptyset$ is a special case of the following more general
notation $X_w$, useful for
referring to subsurfaces.
If $w\in W$ then we define $W_{\!w}$ as the set of binary words
having $w$ as a prefix.  Equivalently, it consists of~$w$ and
the vertices of~$T$ 
that lie below~$w$.  We define $X_w$ as $(P\times W_{\!w})/{\sim}$.
If is easy to see that if $x\in W_{\!w}$ then $P_x\to P\times W_{\!w}\to X_w\to X_\emptyset$
is injective.  So we will regard all $P_x$ and $X_w$ as subsets 
of~$X_\emptyset\sset X$.  In particular,
$X_w$ is a subsurface of~$X$
bounded by the waist of~$P_w$.

Now suppose $g\in\{a,b,c,d\}\sset G$.  We will define an
automorphism $\ghat$ 
of $X$ that is supported on~$X_\emptyset$.
We define it in terms of an automorphism of $P\times W$ that respects~$\sim$.
Namely, if $p\in P$ and $w\in W$, then
\begin{equation}
    \label{EqDefnOfghat}
    \ghat(p,w)=
    \begin{cases}
        (\sigma(p),w)&\hbox{if $w$ is a swap vertex of~$g$}
        \\
        (p,g(w))&\hbox{otherwise.}
    \end{cases}
\end{equation}
After checking that this
respects~$\sim$, we may regard $\ghat$ as an automorphism of~$X$.

\begin{figure}
\includegraphics[width=\picwidth,trim=130 45 60 120,clip=true]{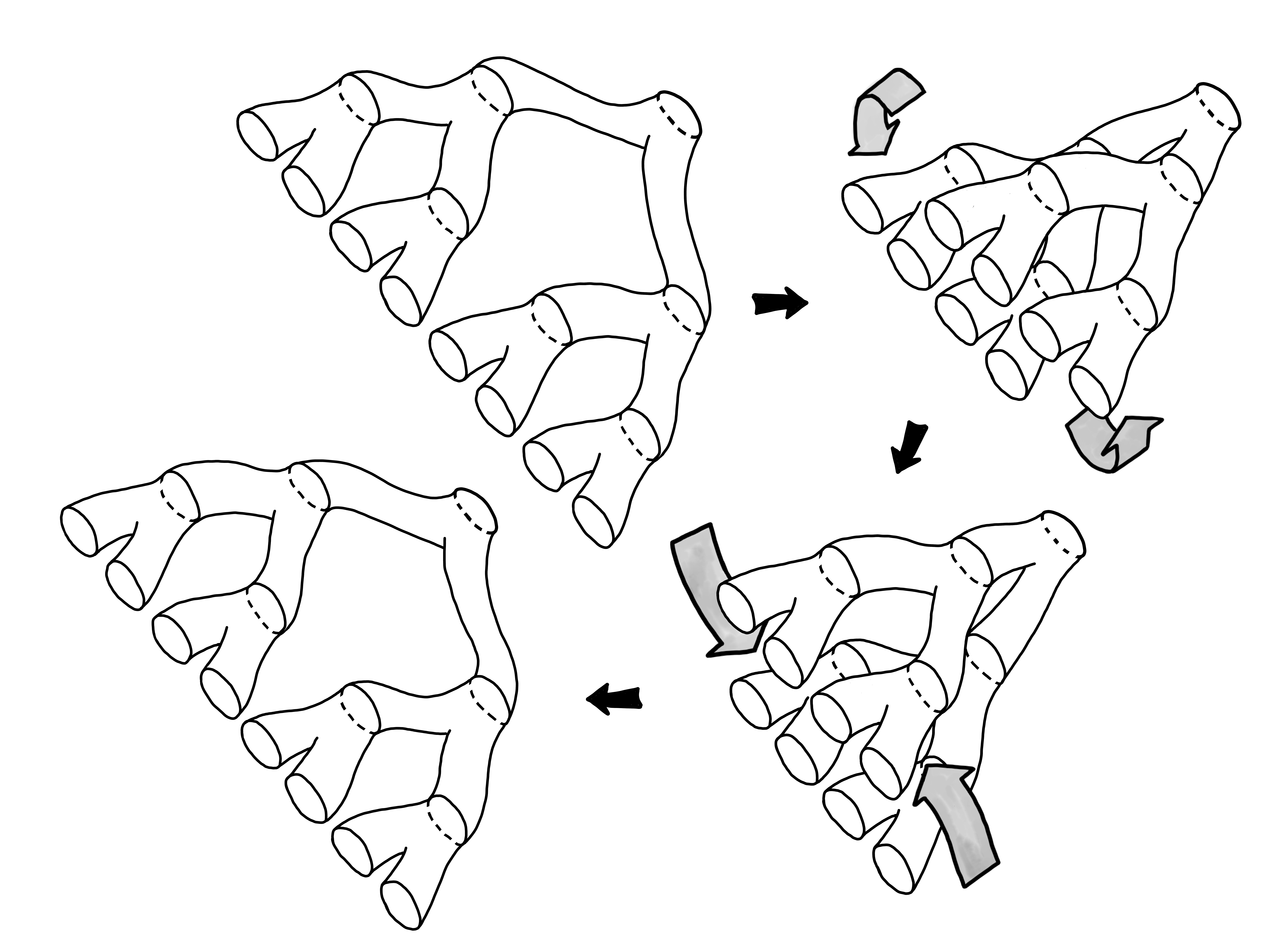}
    \caption{Our construction of~$\ghat$.  Here $g$ is one of 
    Grigorchuk's generators $a,b,c,d$, and
    the surface shown represents $X_w\sset X$, 
    where $w\in W$ is a swap vertex of~$g$.  We define $\ghat$ (on $X_w$) as
    the end result of the pictured isotopy in~$\R^3$.  See the text for details.}
    \label{FigTheTrick}
\end{figure}

A visual version of this construction, entirely optional, appears
in Figure~\ref{FigTheTrick}, which shows $X_w$ for $w$ a swap vertex of~$g$.
The same construction should be carried out simultaneously for every
swap vertex.
The figure shows four frames of an isotopy of $X_w$ inside~$\R^3$,
whose time~$1$ map we take as the definition of~$\ghat$.  Although
$X_w$ is the union of infinitely many pairs of pants, we have only
drawn enough of them to illustrate the idea.  You should imagine 
attaching another $8$ pairs of pants on the left, then attaching
$16$ more pairs, then another~$32$, and so on,
limiting to the Cantor set which is the end space of $X_w$.
The top right pair of pants  is~$P_w$, 
and we imagine it is
made of rubber, deforming as needed 
to accomodate the motion
of  the surfaces $X_{w0}$
and $X_{w1}$ attached to its cuffs, which are made of steel.
The isotopy moves $X_{w0}$ and~$X_{w1}$ rigidly
(ie, by translations) until they trade places.
Strictly speaking,
$X_{w0}$ and $X_{w1}$ should appear identical, so that the time~$1$
map of the isotopy is a self-map of the pictured surface.
But we have drawn them slightly differently,
so the reader can follow their motions more easily.  It is easy
to see that the restriction of $\ghat^2$ to~$P_w$ is
(isotopic to) the product of
the left-handed Dehn twist around the waist and 
the right-handed Dehn twists around the cuffs.

Now we return to the formal development.
We define $\atilde,\btilde,\ctilde,\dtilde\in\Map(X)$
as the mapping classes of $\ahat,\bhat,\chat,\dhat$, and
\begin{align*}
    \Ghat{}=\gend{\ahat,\bhat,\chat,\dhat}\sset\Homeo(X)
    \qquad
    \Gtilde{}=\gend{\atilde,\btilde,\ctilde,\dtilde}\sset\Map(X)
\end{align*}
By construction, $\ahat,\bhat,\chat,\dhat$
permute the $P_w\sset X$ in the
same way that $a,b,c,d\in G$
permute the vertices~$w$ of~$T$.  
It follows that $\ahat,\bhat,\chat,\dhat$
permute the ends of $X_\emptyset$ the same way that $a,b,c,d$
permute the ends of $T$.  The same holds with
$\atilde,\ldots,\dtilde$ in place of $\ahat,\dots,\dhat$,
because the action of $\Homeo(X)$ (or $\Homeo^+(X)$ if $X$ is orientable)
on the ends of~$X$ factors
through $\Map(X)$. 
Also, $\Gtilde$ fixes every end of $X$
not coming from~$X_\emptyset$.
We have shown that the image of~$\Gtilde$, under the
action of $\Map(X)$ on the ends of~$X$, is isomorphic to~$G$.

\begin{lemma}
    \label{LemCantorAbelianKernel}
    The surjection $\Gtilde\to G$ has abelian kernel.
\end{lemma}

\begin{proof}
    Let $U$ be the union of a family of mutually disjoint annular neighborhoods
    of the waists of the $P_w$, with $w$ varying over~$W$. 
    We claim that every $\alpha\in\Ghat$,
    that acts trivially on the end space of~$X$, is
    supported on~$U$ (up to isotopy).  
    Given this, it follows that the 
    kernel of $\Gtilde\to G$ lies in the image of $\Map(U)\to\Map(X)$.
    Then the lemma follows from the fact
    that $\Map(U)$ is a direct product
    of copies of~$\Z$.

    Now we prove the claim.
    Because $\alpha$ permutes the $P_w$, and acts trivially on the ends
    of~$X_\emptyset$, it sends each $P_w$ to itself.
    Considering their intersections
    shows that $\alpha$
    preserves the waist of every~$P_w$.  
    After an isotopy supported in~$X_\emptyset$
    and preserving every~$P_w$, we may
    suppose $\alpha$ fixes each waist pointwise. 
    The mapping class group
    of a pair of pants is generated by the Dehn twists around its
    boundary components.  
    Applying this separately to each~$P_w$ shows that $\alpha$
    is isotopic to an automorphism of~$X$ that is supported in~$U$.
\end{proof}

\begin{lemma}
    \label{LemSolvableKernel}
    Suppose $f:\Gtilde\to G$ is a surjection of groups with solvable kernel.
    Then
    \begin{enumerate}
        \item
            \label{ItemVirtuallySolvable}
            $\Gtilde$ is virtually solvable if and only if~$G$ is.
        \item
            \label{ItemContainsNonabelianFree}
            $\Gtilde$ contains a nonabelian free group if and only if~$G$ does.
    \end{enumerate}
\end{lemma}

\begin{proof}
    \eqref{ItemVirtuallySolvable}
    Obvious.
    \eqref{ItemContainsNonabelianFree}
    First suppose $G$ contains a subgroup freely generated by two elements $x,y$,
     and choose any lifts $\xtilde,\ytilde$ of them in~$\Gtilde$.
     These lifts satisfy no nontrivial relations,
    because if they did then $x,y$  would also.
    So $\xtilde,\ytilde$ generate a nonabelian free subgroup of~$\Gtilde$.
    Now suppose $\Gtilde$ contains a nonabelian free group~$\Ftilde$.  Being nonabelian
    and free, $\Ftilde$ lacks normal solvable subgroups, so it meets $\Ker(f)$
    trivially.  
    Therefore $f(\Ftilde)\iso\Ftilde$ is a  nonabelian free subgroup of~$G$.
%
\end{proof}

\begin{lemma}
    \label{LemCantorTitsAlternativeFails}
    With $X$ as above, $\Map(X)$ does not satisfy the Tits alternative.
\end{lemma}

\begin{proof}
    It is enough to show that $\Gtilde\sset\Map(X)$ is not virtually solvable
    and contains no nonabelian free group.  By
    Lemmas \ref{LemCantorAbelianKernel}
    and \ref{LemSolvableKernel} it is enough to 
    prove this with $G$ in place of~$\Gtilde$. We did this
    in Lemma~\ref{LemGrigorchuk}.
\end{proof}

\section{If \texorpdfstring{$X$}{X} has infinitely many handles or punctures}
\label{SecPuncturesOrHandles}

\noindent
In this section we consider a surface $X$ with infinite
genus, or infinitely many punctures.
Again we will build
a subgroup~$\Gtilde$ of $\Map(X)$ that is not virtually solvable and
contains no nonabelian free group.  
The method is a variation on the previous section.
We will focus on the
infinite-genus case, and then modify the
argument for the case of infinitely many punctures.

Suppose $X$ is a surface with infinite genus.  Take
$X_1$ to be a compact subsurface with genus~$2$ and one boundary component.
Then take $X_2$ to be a compact subsurface of $X-X_1$ with genus~$4$
and one boundary component.  Then  take
$X_3$ to be a compact subsurface of $X-(X_1\cup X_2)$ with genus~$8$
and one boundary component.  Continuing in this
fashion, we obtain an infinite
sequence $X_n$ of mutually disjoint compact subsurfaces
of~$X$, where
$X_n$ has genus~$2^n$ and one boundary component.  
As suggested by Figure~\ref{FigPantsAndFeet}, 
we express each $X_n$ as
the union of $2^n-1$ pairs of pants, arranged like 
a binary tree, and $2^n$ handles.

\begin{figure}
\includegraphics[width=\picwidth,trim=330 420 200 420,clip=true]{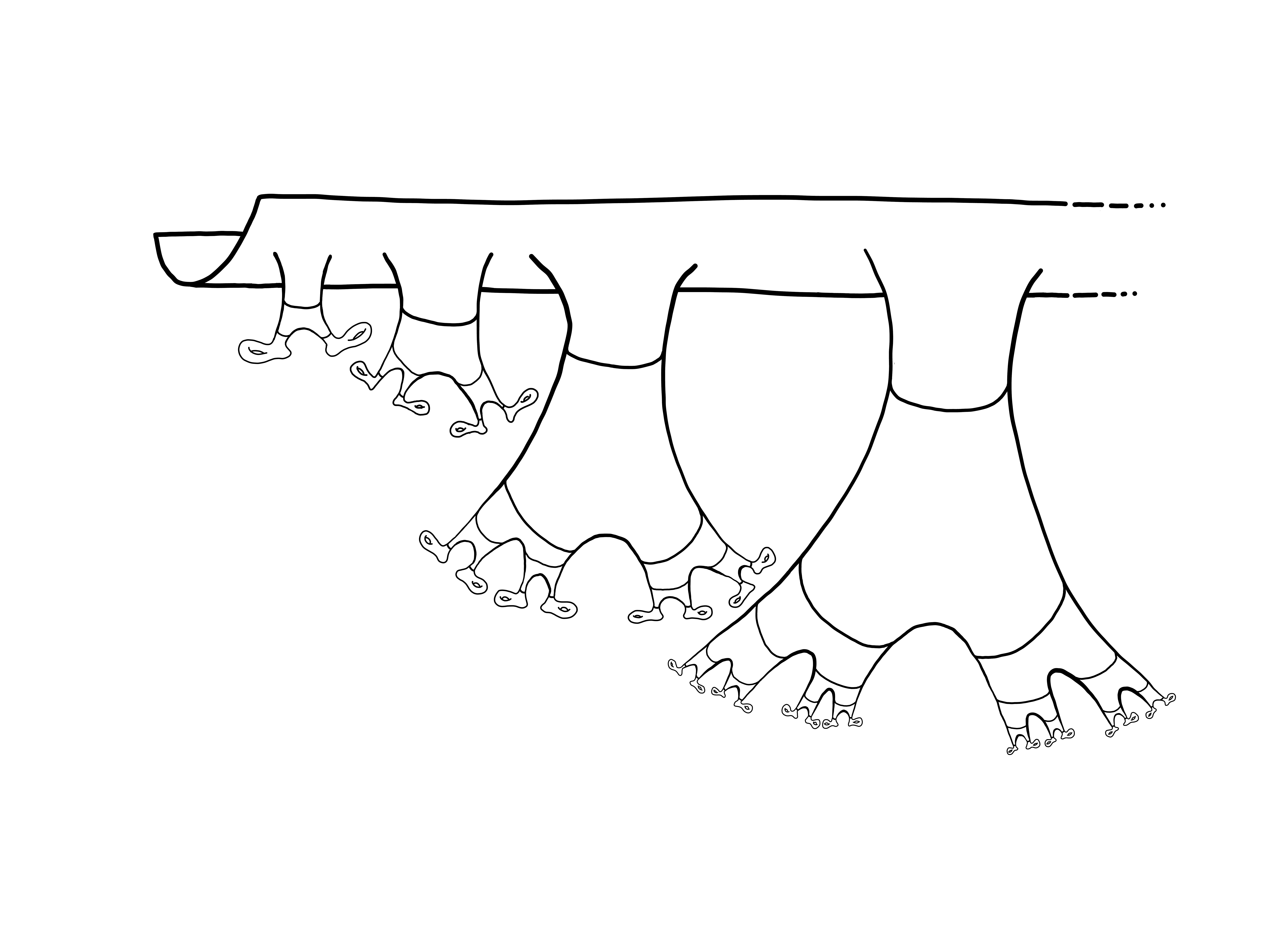}
    \caption{The subsurfaces $X_1,X_2,X_3,X_4,\dots$ of~$X$.}
    \label{FigPantsAndFeet}
\end{figure}

To express this precisely, recall 
$W$, $P$ and $P_w$ (with $w\in W$) from
Section~\ref{SecCantorSet}.
For each integer~$n$ we write $W_n\sset W$ for the set of
binary words of length~$n$, $W_{\!<n}$ for those of length${}<n$,
and $W_{\!\leq n}$ for those of length${}\leq n$.

We also fix a compact surface $S$ (for \emph{S}hoe) 
with genus~$1$ 
and one boundary component.  We will call $\partial S$ the rim of~$S$;
a shoemaker would say collar, but that already has
a meaning in topology.
We fix a homeomorphism from the rim to the circle~$S^1$.

For each~$n\geq1$, we identify $X_n$ with the quotient of 
$(P\times W_{\!<n})\cup(S\times W_n)$ by the following equivalence
relation~$\sim$, similar to the one in
Section~\ref{SecCantorSet}.  We abbreviate $S\times\{w\}$
as $S_w$, for each $w\in W_n$.  We refer to the $S_w$ as the \defn{shoes}.  First,
for each $w\in W_{\!<n-1}$, we glue the cuffs of 
$P_w$ to the waists of $P_{w0}$ and $P_{w1}$.  Formally,
we
declare that $(l,w)\sim(l',w0)$
and $(r,w)\sim(r',w1)$, where $l$ resp.\ $r$ lies in the left
resp.\ right cuff of~$P$, and $l'$ resp.\ $r'$ is the corresponding
point on the waist of~$P$.  Second, for each $w\in W_{n-1}$
we glue the cuffs of $P_w$ to the rims of $S_{w0}$ and $S_{w1}$.
Formally, we declare that
$(l,w)\sim(l',w0)$ and $(r,w)\sim(r',w1)$, where
$l,r$ are as before, but now $l',r'$ are the corresponding
points on the rim of~$S$.

For each $n\geq1$, and each $w\in W_{\!<n}$, we have introduced a
subsurface $P_w$ of~$X_n$.  Since all the $X_n$ lie in~$X$,
we have introduced infinitely many subsurfaces of~$X$, all called~$P_w$.
To distinguish them when necessary, we write $P_{n,w}$ for
the one that lies in~$X_n$.  The same issue does not arise for
the shoes $S_w$: if $w\in W$ has length~$n$, 
then only $X_n$ contains $S_w$.

For $g\in\{a,b,c,d\}\sset G$, we now define an automorphism~$\ghat$
of~$X$.  It is the identity outside the $X_n$, and for each fixed $n$ we 
define $\ghat$ on~$X_n$ as follows.  The idea is to act on $X_n$
``in the same way that $g$ acts on the top $n+1$ levels of the tree~$T$''. 
Formally, we define $\ghat$ as an automorphism of 
$(P\times W_{\!<n})\cup(S\times W_n)$ that respects~$\sim$.
For each $w\in W_{\!<n}$ we define $\ghat$ on $P_w$ 
by \eqref{EqDefnOfghat}.
And for $w\in W_n$ we define
$\ghat$ on $S_w$ by $\ghat(p,w)=(p,g(w))$.  After checking that
this respects~$\sim$, we may regard $\ghat$ as defined on~$X_n$.
This completes the construction of $\ghat\in\Homeo(X)$.
If $X$ is orientable then clearly $\ghat$ lies in $\Homeo^+(X)$.

We define $\atilde,\btilde,\ctilde,\dtilde\in\Map(X)$
as the mapping classes of $\ahat,\bhat,\chat,\dhat$, and
\begin{align*}
    \Ghat{}=\gend{\ahat,\bhat,\chat,\dhat}\sset\Homeo(X)
    \qquad
    \Gtilde{}=\gend{\atilde,\btilde,\ctilde,\dtilde}\sset\Map(X)
\end{align*}

\begin{lemma}
    \label{LemGenusSurjectionToG}
    The image of $\Gtilde$ in $\Aut(H_1(X))$ is isomorphic to~$G$.
\end{lemma}

\begin{proof}
    Because $\Ghat$ fixes $X-\bigcup_n X_n$
    pointwise, it acts trivially on the homology classes 
    supported there.  So it is enough to work out the action
    of~$\Gtilde$ on 
    \begin{equation*}
        H_1\bigl(\bigcup_{n\geq1} X_n\bigr)
        =
        \bigoplus_{n\geq1} H_1(X_n)
        =
        \bigoplus_{n\geq1}\, \bigoplus_{w\in W_n} H_1(S_w)
        =
        \bigoplus_{w\in W-\{\emptyset\}}\kern-.8em H_1(S_w)
    \end{equation*}
    By construction, $\Ghat$ permutes the shoes, hence
    the summands on the right.  In particular, an element
    of $\Ghat$ that preserves a summand $H_1(S_w)$ must carry
    the corresponding shoe $S_w$ to itself.  
    Furthermore, the definition of the generators of~$\Ghat$ on the shoes
    implies: an element of $\Ghat$ that preserves some~$S_w$ must
    act trivially on it.  Therefore: any element of $\Ghat$,
    that preserves every summand $H_1(S_w)$, must act trivially on
    every $S_w$, hence trivially on~$H_1(X)$.

    Another way to say this is that the image of $\Ghat$ in $\Aut(H_1(X))$
    is the same as the image of $\Ghat$ in the group of permutations of
    the shoes.
    By construction, $\ahat,\bhat,\chat,\dhat$
    permute the shoes in the
    same way that $a,b,c,d\in G$ permute the words $w\in W-\{\emptyset\}$.
    So the image of $\Ghat$ in $\Aut(H_1(X))$ is a copy of~$G$.
    The action on $H_1(X)$ factors through $\Map(X)$, so the
    same is true with $\Gtilde$ in place of $\Ghat$.
\end{proof}

\begin{lemma}
    \label{LemGenusAbelianKernel}
    The surjection $\Gtilde\to G$ has abelian kernel.
\end{lemma}

\begin{proof}
    We will write $f$ for the surjection $\Ghat\to G$.
    We must show that the image of $\Ker(f)$ in~$\Gtilde$ is abelian.
    So suppose $\alpha\in\Ker(f)$.  The previous proof shows that
    $\alpha$ acts trivially on every shoe~$S_w$.
    Together with the fact that $\alpha$ permutes the $P_{n,w}$, 
    this shows that $\alpha$ preserves  every $P_{n,w}$.

    From here one can follow
    the proof of Lemma~\ref{LemCantorAbelianKernel}.  We take
    $U$ to be the union of disjoint annular neighborhoods of
    the waists of the $P_{n,w}$ (with $n\geq1$ and $w\in W_{\!<n}$) and 
    the rims of the $S_w$ (with $w\in W-\{\emptyset\}$).
    Arguing as for Lemma~\ref{LemCantorAbelianKernel} shows that
     $\alpha$ is isotopic
    to a homeomorphism supported on~$U$.  
    Because $\Map(U)$ is abelian, it follows that $\Ker(f)$ has
    abelian image in~$\Gtilde$.
\end{proof}

\begin{lemma}
    \label{LemGenusTitsAlternativeFails}
    The mapping class group of an infinite-genus surface 
    does not satisfy the Tits Alternative.
\end{lemma}

\begin{proof}
    Mimic the proof of Lemma~\ref{LemCantorTitsAlternativeFails}, 
    using Lemma~\ref{LemGenusAbelianKernel} in
    place of Lemma~\ref{LemCantorAbelianKernel}.  
\end{proof}

To deal with punctures in place of handles,
we first define punctures.  Let $E$ be the end space of~$X$; 
we equip $X\cup E$ with
its standard topology.  
For us, A \defn{puncture} of~$X$ means an end $e$ for which
there exists an embedding of the $2$-disk into $X\cup E$, that sends
the origin to~$e$ and sends no other point into~$E$.

\begin{lemma}
    \label{LemPuncturesTitsAlternativeFails}
    If $X$ has
 infinitely many
    punctures, then its mapping class group
    does not satisfy the Tits Alternative.
\end{lemma}

\begin{proof}[Proof sketch.]
    Choose an infinite sequence of distinct ends~$e_n$, and 
    disks in~$X\cup E$ centered at them, as above.  By shrinking
    the $n$th disk, we may
    suppose it is disjoint from its predecessors.
    So any two of the disks are disjoint.
    With this preparation, it is easy to construct a sequence 
    of disjoint closed sets $X_1,X_2,\dots$ in~$X$, such that $X_n$ is
    a disk with $2^n$ punctures.  We decompose $X_n$ as in the
    infinite-genus case, except that we use  once-punctured disks
    in place of the shoes. 
    The surface
    automorphisms
    $\ahat,\bhat,\chat,\dhat$ and the group $\Ghat$ they generate
    are defined as before.  One
    replaces Lemma~\ref{LemGenusSurjectionToG}
    by an analysis of how $\Ghat$ permutes the ends of~$X$.  This is 
    simpler than Lemma~\ref{LemGenusSurjectionToG}
    and similar to what we did in Section~\ref{SecCantorSet}.
    The only difference is that the ends of $\bigcup_n X_n$ are indexed
    by the vertices of~$T$, whereas there 
    the ends of $X_\emptyset$ were indexed by
    the ends of~$T$.  The rest of the proof is the same as in the
    infinite-genus case.
\end{proof}

\section{The main results}
\label{SecMainTheorem}

\noindent
Theorem~\ref{ThmMain} is the union of Lemmas \ref{LemCantorTitsAlternativeFails},
\ref{LemGenusTitsAlternativeFails} and~\ref{LemPuncturesTitsAlternativeFails}.
For Corollary~\ref{CorMain}, suppose $X$ has
infinite type but only finitely many boundary
components. The
theorem below shows that $X$ has one of the features
\eqref{CaseGenus}--\eqref{CaseCantor}
in Theorem~\ref{ThmMain}.  Quoting that theorem shows that
$\Map(X)$ does not satisfy the Tits Alternative.
Lanier and Loving used the same trichotomy  
to construct the surfaces
mentioned in the introduction, whose mapping class groups
fail to satisfy
the \emph{strong} Tits Alternative.

\begin{theorem}
    \label{ThmTrichotomy}
    Suppose $X$ is an infinite-type surface with only finitely
    many boundary components.  Then $X$ has infinite genus,
    or infinitely many punctures, or contains a closed subset
    homeomorphic to a disk with a Cantor set removed from its
    interior.
\end{theorem}

We call a surface~$X$ \defn{planar} if  every circle 
embedded in $X-\partial X$
is separating.

\begin{lemma}
    \label{LemEndRecognition}
    Suppose $X$ is a planar surface and $E$ its space of ends.
    Suppose $e$ is an isolated point of~$E$ and lies outside the
    closure of $\partial X$ (in $X\cup E$).  Then $e$ is a puncture.
\end{lemma}

\begin{proof}
Straightforward.
\end{proof}

\begin{proof}[Proof of Theorem~\ref{ThmTrichotomy}]
    The arguments
    are like those of Richards \cite{Richards} in his proof of Ker\'ekj\'art\'o's
    theorem classifying noncompact surfaces.  We would like to quote his results,
    but they do not allow boundary.  A generalization due to
    Prishlyak and Mischenko \cite{PrishlyakMischenko} allows boundary, but
    I could not follow it.

    If $X$ is nonorientable, then we can find an embedded M\"obius band $M_1$.
    If $X-M_1$ is nonorientable, then it contains an embedded M\"obius band
    $M_2$.  Repeating this process with $X-(M_1\cup M_2)$, and so on, we find 
    a sequence of mutually disjoint embedded M\"obius bands. Suppose
    first that this construction does not terminate, so 
    the sequence is infinite.    The
    connected sum of three crosscaps is the same as the connected sum of a
    torus and a single crosscap,  so $X$ has infinite genus.
    
    This leaves the case that the sequence terminates.  
    Then $X$ is the connected sum of an orientable surface 
    $X_1$ and finitely many crosscaps.  If the genus of
    $X_1$ is infinite 
    then we are done, so suppose it is finite.
    Gathering the handles and crosscaps together, we see that 
    $X$ is the connected sum of a planar surface $X_2$ and
    a closed surface.  Because $X$ has infinite type, $X_2$ does too.

    We write $E$ for the end space of~$X$, which is also the
    end space of~$X_2$.   It is infinite, or else
    $X_2$ would have finite type.
    Its intersection
    with $\overline{\partial X}$ (the closure taken in $X\cup E$) is
    finite, because $X$ has finitely many boundary
    components, each with at most two ends.  So $E-\overline{\partial X}$
    is infinite.   The finiteness of
    $E\cap\overline{\partial X}$ also shows that
    a point of $E-\overline{\partial X}$ is isolated in
    $E-\overline{\partial X}$ if and only if it is isolated in~$E$.
    So we may speak unambiguously of isolated ends.  

    If there are infinitely many isolated points of $E-\overline{\partial X}$,
    then we are done because Lemma~\ref{LemEndRecognition} shows that
    each is a puncture.  So suppose otherwise.
    Because $E$ is infinite,  $E-\overline{\partial X}$ has a
    non-isolated point~$e$.  Choose an embedded
    circle $C\sset X_2$ that separates $e$ from
    $\partial X_2$ and the isolated ends. 
    Of the two components of $X_2-C$, one has $e$ among its ends.
    Write $D$ for the union of that component and~$C$, and
    write $E_0$ for the end space of~$D$.
    Because $E_0$ has no isolated points, it
    is perfect and therefore a Cantor set.  Because $X_2$ is planar,
    a standard subdivision-into-pairs-of-pants argument shows that $D$ is a disk 
    with a Cantor set removed from its interior.
\end{proof}

If $X$ has infinitely many boundary components, then our constructions
may fail.  The reason is that the homeomorphisms used to define
$\Map(X)$ must fix $\partial X$ pointwise, so that they cannot permute
the components of~$\partial X$.  
But even if one allowed permutations in some way,
some surfaces would still escape our methods:

\begin{example}
    Let $D_{n\in\Z}$ be the closed disk in~$\R^2$ with radius~$1/3$
    and center $(n,0)$.  Let $C_n$ be a Cantor set in $\partial D_n$.
    Let $X$ be the plane, minus the union of the~$C_n$ and
    the interiors of the~$D_n$.  Then $X$ has countably many boundary
    components~$B_{k\geq1}$, each a copy of~$\R$.  From $X$ remove
    $k$ points from each~$B_k$, and write~$Y$ for what remains.  Then
    every self-homeomorphism of~$Y$ sends every component of~$\partial Y$
    to itself.  
    (The key step is that every self-homeomorphism of~$Y$ sends 
    every $B_k$ to itself.  To see this, say that two components
    of~$\partial Y$ \emph{abut} if some end of $Y$ lies in both
    their closures.  Under the equivalence relation this generates,
    two components of $\partial Y$ are equivalent if and only if they
    lie in the same~$B_k$.  Every self-homeomorphism preserves the
    unique equivalence class of each size~$k>1$, hence every $B_k$.)
    It seems likely that the central quotient of $\Map(Y)$ is the 
    end-preserving subgroup of $\Map(Y-\partial Y)$, ie 
    a version of the pure braid group on infinitely many strands.
\end{example}

\end{document}